\documentclass[a4,12pt]{amsart}
\usepackage{amssymb,amstext,amsmath,amscd,amsthm,amsfonts,enumerate,graphicx,latexsym,color}
\usepackage[all]{xy}
\newtheorem{thm}{Theorem}[section]
\newtheorem*{thm*}{Theorem}

\newtheorem{prop}[thm]{Proposition}
\theoremstyle{definition}
\newtheorem{dfn}[thm]{Definition}
\newtheorem*{dfn*}{Definition}

\newtheorem{ques}[thm]{Question}

\theoremstyle{remark}

\newtheorem*{ac}{Acknowledgments}

\newtheorem*{claim*}{Claim}

\numberwithin{equation}{thm}
\def\p{\mathfrak{p}}
\def\X{\mathcal{X}}
\def\G{\mathsf{G}}
\def\S{\mathsf{S}}
\def\tr{\operatorname{\mathsf{Tr}}}
\def\tf{\mathsf{TF}}
\def\min{\mathsf{min}}
\def\Ext{\mathsf{Ext}}
\def\Hom{\mathsf{Hom}}
\def\mod{\operatorname{\mathsf{mod}}}
\def\syz{\mathsf{Syz}}

\def\dim{\operatorname{\mathsf{dim}}}
\def\grade{\operatorname{\mathsf{grade}}}
\def\depth{\operatorname{\mathsf{depth}}}
\def\h{\operatorname{\mathsf{ht}}}
\title{When are $n$-syzygy modules $n$-torsionfree?}
\address{Graduate School of Mathematics, Nagoya University, Furocho, Chikusaku, Nagoya, Aichi 464-8602, Japan}
\author{Hiroki Matsui} 
\email{m14037f@math.nagoya-u.ac.jp}
\author{Ryo Takahashi}
\address{Graduate School of Mathematics, Nagoya University, Furocho, Chikusaku, Nagoya, Aichi 464-8602, Japan}
\email{takahashi@math.nagoya-u.ac.jp}
\urladdr{http://www.math.nagoya-u.ac.jp/~takahashi/}
\author{Yoshinao Tsuchiya} 
\address{Graduate School of Mathematics, Nagoya University, Furocho, Chikusaku, Nagoya, Aichi 464-8602, Japan}
\email{m15034w@math.nagoya-u.ac.jp}
\begin{document}
\baselineskip=15pt
\begin{abstract}
Let $R$ be a commutative noetherian ring.
We consider the question of when $n$-syzygy modules over $R$ are $n$-torsionfree in the sense of Auslander and Bridger.
Our tools include Serre's condition and certain conditions on the local Gorenstein property of $R$.
Our main result implies the converse of a celebrated theorem of Evans and Griffith.
\end{abstract}
\maketitle
\section{Introduction}

The notion of $n$-torsionfree modules was introduced by Auslander and Bridger \cite{ABr} to treat the theory of torsionfree modules over integral domains in general situations.
They proved that all $n$-torsionfree modules are $n$-syzygy.
It is thus natural to ask when the converse holds, that is:

\begin{ques}\label{q}
When are $n$-syzygy modules $n$-torsionfree?
\end{ques}

This question has been investigated by several authors so far, and among other things, the following result of Evans and Griffith is celebrated; see \cite[Theorem 3.8]{EG} and also \cite[Lemma 1.3]{LW}.

\begin{thm}[Evans--Griffith] \label{eg}
Let $R$ be a commutative noetherian ring.
Let $n\ge0$ be an integer.
Suppose that $R$ satisfies $(\G_{n-1})$ and $(\S_n)$.
Then for an $R$-module $M$ one has:
$$
\text{$M$ is $n$-syzygy}\ \Longleftrightarrow\ 
\text{$M$ is $n$-torsionfree}\ \Longleftrightarrow\ 
\text{$M$ satisfies $(\S_n)$.}
$$
\end{thm}

Here, $(\S_n)$ is Serre's condition, while $(\G_{n-1})$ is a certain condition on the local Gorenstein property of a commutative noetherian ring.
See Definition \ref{def} in the next section for the definitions of the notions appearing in the theorem.
Recently, Goto and Takahashi \cite{GT} have proved that the converse of Theorem \ref{eg} holds under the additional assumption that $R$ is local.

The purpose of this paper is to proceed with the study of Question \ref{q}.
Our main result given in the next section includes the following.

\begin{thm}\label{it}
Let $R$ be a commutative noetherian ring.
Let $n\ge0$ be an integer.
Then the following three conditions are equivalent.
\begin{enumerate}[\rm(1)]
\item
The ring $R$ satisfies $(\G_{n-1})$ and $(\S_n)$.
\item
For an $R$-module $M$ one has:
$$
\text{$M$ is $n$-syzygy}\ \Longleftrightarrow\ 
\text{$M$ is $n$-torsionfree}\ \Longleftrightarrow\ 
\text{$M$ satisfies $(\S_n)$.}
$$
\item
For an $R$-module $M$ one has:
$$
\text{$M$ is $(n+1)$-syzygy}\ \Longleftrightarrow\ 
\text{$M$ is $(n+1)$-torsionfree.}
$$
\end{enumerate}
\end{thm}

It looks surprising that the conditions (2) and (3) are equivalent, both in the sense that (3) does not contain any requirements on Serre's condition but (2) does, and in the sense that (2) is on the number $n$ while (3) is on $n+1$.
Furthermore, this theorem implies that the converse of Theorem \ref{eg} due to Evans and Griffith holds true even if $R$ is non-local, which is not covered by \cite{GT}.

In the next section, we will first recall the precise definitions of the notions appearing above, and then state our main result and prove it.

\section{The main result}

Throughout this section, let $R$ be a commutative noetherian ring.
Denote by $\mod R$ the category of finitely generated $R$-modules.
We begin with recalling several definitions.

\begin{dfn}\label{def}
\begin{enumerate}[(1)]
\item
A full subcategory $\X$ of $\mod R$ is said to be {\it closed under extensions} provided that for each exact sequence $0 \to L \to M \to N \to 0$ in $\mod R$, if $L$ and $N$ are in $\X$, then so is $M$.   
\item
Let $P_1 \xrightarrow{f} P_0 \to M \to 0$ be an exact sequence in $\mod R$ with $P_0,P_1$ projective.
The {\em (Auslander) transpose} of $M$ is defined as the cokernel of the homomorphism
$$
\Hom_R(f, R): \Hom_R(P_0, R) \to \Hom_R(P_1, R)
$$ 
and denoted by $\tr M$.
\item
Let $n\ge0$ be an integer.
Let $M$ be a finitely generated $R$-module.
We say that
\begin{enumerate}[(a)]
\item
$M$ is {\it $n$-syzygy} if there exists an exact sequence $0 \to M \to P_{n-1} \to P_{n-2} \to \cdots \to P_1 \to P_0$ in $\mod R$ with each $P_i$ projective.
\item
$M$ is {\it $n$-torsionfree} if $\Ext_R^i(\tr M, R)=0$ for all $1 \le i \le n$.
\item
$M$ satisfies {\it Serre's condition $(\S_n)$} if the inequality $\depth M_{\p} \ge \min \{n, \h \p \}$ holds for all prime ideals $\p$ of $R$.
\end{enumerate}
\item
For an integer $n\ge-1$ the ring $R$ is said to satisfy the condition $(\G_n)$ (resp. $(\widetilde\G_n)$) if the local ring $R_{\p}$ is Gorenstein for every prime ideal $\p$ of $R$ with $\h \p$ (resp. $\depth R_\p$) is at most $n$.
\end{enumerate}
\end{dfn}

We denote by $\syz_n(R)$, $\tf_n(R)$ and $\S_n(R)$ the full subcategories of $\mod R$ consisting of $n$-syzygy modules, $n$-torsionfree modules and modules satisfying $(\S_n)$, respectively.
Here are some remarks and basic properties which we will use.

\begin{prop}\label{prop}
\begin{enumerate}[\rm(1)]
\item
The transpose is uniquely determined up to projective summands.
\item
Every $R$-module is both $0$-syzygy and $0$-torsionfree, and satisfies $(\S_0)$.
\item
The ring $R$ always satisfies $(\G_{-1})$ and $(\widetilde{\G}_{-1})$.
\item
The finitely generated projective $R$-modules are both $n$-syzygy and $n$-torsionfree for all integers $n\ge0$.
\item
Let $n$ be a non-negative integer.
The following assertions hold.
\begin{enumerate}[\rm(a)]
\item
If $R$ satisfies $(\S_n)$, then so do all the $n$-syzygy modules.
\item
The full subcategory $\S_n(R)$ of $\mod R$ is closed under extensions.
\end{enumerate}
\end{enumerate}
\end{prop}

\begin{proof}
The statement (1) is follows from \cite[Proposition (2.6)(c)]{ABr}, while (2) and (3) are evident.
For a finitely generated projective $R$-module $P$, the sequence $0\to P\xrightarrow{=}P\to0\to\cdots\to0$ is exact, and $\tr P=0$.
This shows (4).
The two assertions of (5) can be proved by using the depth lemma.
\end{proof}

The following theorem is the main result of this paper, including Theorem \ref{it}, whence including the converse of Theorem \ref{eg}.
(Note that Theorem \ref{it} asserts the equivalences $(2)\Leftrightarrow(4)\Leftrightarrow(6)$.)

\begin{thm}\label{main}
Let $R$ be a commutative noetherian ring.
Let $n$ be a non-negative integer.
Consider the following eight conditions.
\begin{enumerate}[\ \quad\rm(1)]
\item
$R$ satisfies $(\widetilde{\G}_{n-1})$.
\item
$R$ satisfies $(\G_{n-1})$ and $(\S_n)$. 
\item
$\tf_n(R)=\S_n(R)$.
\item
$\tf_n(R)=\syz_n(R)=\S_n(R)$.
\item
$\tf_{n}(R)=\syz_n(R)$, and $\syz_n(R)$ is closed under extensions.
\item
$\tf_{n+1}(R)=\syz_{n+1}(R)$.
\item
$\syz_n(R)=\S_n(R)$.
\item
$R$ satisfies $(\S_n)$, and $\syz_n(R)$ is closed under extensions.
\end{enumerate}
Then the implications $(1)\Leftrightarrow(2)\Leftrightarrow(3)\Leftrightarrow(4)\Leftrightarrow(5)\Leftrightarrow(6)\Rightarrow(7)\Rightarrow(8)$ hold.
If $R$ is local, all the eight conditions are equivalent.
\end{thm}

\begin{proof}
It is obvious that the implications $(3)\Leftarrow(4)\Rightarrow(7)$ hold.
The implication $(2)\Rightarrow(4)$ is nothing but Theorem \ref{eg}.
The equivalence $(1) \Leftrightarrow (6)$ follows from \cite[Theorem (2.17) and Proposition (4.21)]{ABr}.
Note by Proposition \ref{prop}.4 that each of the conditions (3) and (7) implies that the $R$-module $R$ belongs to $\S_n(R)$, i.e., $R$ satisfies $(\S_n)$.
Thus the implication $(7)\Rightarrow(8)$ follows by Proposition \ref{prop}.5(b).
The combination of Proposition \ref{prop}.5(a) with \cite[Theorem (2.17)]{ABr} shows the inclusions
$$
\tf_n(R)\subseteq\syz_n(R)\subseteq\S_n(R)
$$
if $R$ satisfies $(\S_n)$, which yields the implication $(3) \Rightarrow (4)$.
Proposition \ref{prop}.5(b) also gives $(4) \Rightarrow (5)$.
When $R$ is local, the equivalences $(2) \Leftrightarrow (7) \Leftrightarrow (8)$ follow from \cite[Theorem B]{GT}.
Consequently, it suffices to prove that the implications $(1) \Rightarrow (2)$ and $(5) \Rightarrow (1)$ hold.

Assume that $R$ satisfies $(\widetilde{\G}_{n-1})$.
Then $R$ satisfies $(\G_{n-1})$ since for each prime ideal $P$ of $R$ it holds that $\h P=\dim R_P\ge \depth R_P$.
Fix a prime ideal $\p$ of $R$.
If $R_\p$ has depth at least $n$, then we have $\depth R_\p\ge n\ge\min\{n,\h\p\}$.
Suppose that $R_\p$ has depth at most $n-1$.
Then $R_\p$ is Gorenstein by assumption, and in particular it is Cohen--Macaulay.
Hence we get
$$
\depth R_{\p} = \dim R_\p= \h \p \ge \min\{n, \h \p\}.
$$
Therefore the inequality $\depth R_\p\ge\min\{n,\h\p\}$ holds for all prime ideals $\p$ of $R$, that is, $R$ satisfies $(\S_n)$.
The implication $(1)\Rightarrow(2)$ now follows.

Let us assume that the equality $\tf_n(R)=\syz_n(R)$ holds and that $\syz_n(R)$ is closed under extensions.
Let $\p$ be a prime ideal $\p$ of $R$ such that $R_\p$ has depth at most $n-1$.
Using the equality $\tf_n(R)=\syz_n(R)$ and \cite[Proposition (2.26) and Corollary (4.18)]{ABr}, we obtain the inequality
$$
\grade_{R} \Ext_R^i(R/\p, R) \ge i-1
$$
for all $1 \le i \le n$.
From the assumption that $\syz_n(R)$ is closed under extensions and \cite[Theorem A]{GT}, it follows that the $R$-module $\Ext_R^n(R/\p, R)$ has grade at least $n$.
Hence
$$
\Ext_R^j (\Ext_R^n(R/\p, R), R)=0
$$
for all integers $j <n$.
Localization at $\p$ gives rise to
$$
\Ext_{R_{\p}}^j (\Ext_{R_{\p}}^n(\kappa(\p), R_{\p}), R_{\p})=0,
$$ 
where $\kappa(\p)$ denotes the residue field of the local ring $R_\p$.
If $\Ext_{R_{\p}}^n(\kappa(\p), R_{\p})$ does not vanish, then it contains $\kappa(\p)$ as a direct summand, which implies $\Ext_{R_{\p}}^j(\kappa(\p), R_{\p}) = 0$ for all $j < n$.
This means that $R_\p$ has depth at least $n$, which contradicts our choice of $\p$.
Thus, we must have $\Ext_{R_{\p}}^n(\kappa(\p), R_{\p}) = 0$.
Applying \cite[Theorem 1.1]{FFGR}, we observe that the local ring $R_{\p}$ is Gorenstein.
Now the implication $(5) \Rightarrow (1)$ follows, and the proof of the theorem is completed.
\end{proof}

\begin{ac} 
The authors are grateful to Tokuji Araya and Kei-ichiro Iima for their various comments. 
\end{ac}

\end{document}